\documentclass[
a4paper,
11pt, 
DIV=15,
toc=bibliography, 
headsepline, 
parskip=half-,
abstract=true,
]{scrartcl}

\usepackage[english]{babel}
\usepackage[utf8]{inputenc}
\usepackage[T1]{fontenc}
\usepackage[english,cleanlook]{isodate}
\usepackage{lmodern}
\usepackage[babel=true]{microtype}
\usepackage{amsmath}
\usepackage{mathtools}
\usepackage{amssymb}
\usepackage{amsthm}
\usepackage{scrlayer-scrpage} 
\usepackage{enumitem}
\usepackage{tikz} 
\usepackage{tikz-3dplot} 
\usetikzlibrary{plotmarks}

\ihead{A.\,Carlotto, M.\,B.\,Schulz, D. \,Wiygul}

\ohead{\csname @title\endcsname}
\setkomafont{pageheadfoot}{\footnotesize}

\providecommand{\R}{\mathbb{R}}
\providecommand{\N}{\mathbb{N}}

\providecommand{\Sp}{\mathbb{S}}

\providecommand{\B}{\mathbb{B}}
\providecommand{\st}{\, :\ }

\providecommand{\R}{\mathbb{R}}
\providecommand{\C}{\mathbb{C}}
\providecommand{\N}{\mathbb{N}}

\providecommand{\Sp}{\mathbb{S}}

\providecommand{\B}{\mathbb{B}}
\providecommand{\st}{\, :\ }
\providecommand{\hsiang}{H}
\providecommand{\torgrp}{G}
\providecommand{\football}{F}
\providecommand{\ccone}{C}
\providecommand{\dist}{\rho}
\providecommand{\maxdst}{R}
\providecommand{\jop}{J}
\providecommand{\alencar}{A}

\providecommand{\dir}{{\mathrm{D}}}
\providecommand{\neu}{{\mathrm{N}}}
\providecommand{\rob}{{\mathrm{R}}}
\providecommand{\dbdy}{\partial_\dir}
\providecommand{\nbdy}{\partial_\neu}
\providecommand{\rbdy}{\partial_\rob}
\providecommand{\inttext}{\mathrm{int}}
\providecommand{\exttext}{\mathrm{ext}}
\providecommand{\intbdy}{\partial_\inttext}
\providecommand{\extbdy}{\partial_{\exttext}}
\providecommand{\dint}[1]{#1^{\dir_\inttext}}
\providecommand{\nint}[1]{#1^{\neu_\inttext}}
    
\DeclarePairedDelimiter\abs{\lvert}{\rvert}

\DeclarePairedDelimiter\sk{\langle}{\rangle}
\DeclarePairedDelimiter\interval{]}{[}

\DeclarePairedDelimiter\IntervaL{[}{]}

\DeclareMathOperator{\Ogroup}{O}
\DeclareMathOperator{\SOgroup}{SO}
\DeclareMathOperator{\Ugroup}{U}

\DeclareMathOperator{\Ind}{Ind}
\DeclareMathOperator{\Area}{area}
\DeclareMathOperator{\genus}{genus}

\DeclareMathOperator{\II}{I\hspace*{-.5pt}I}
\DeclareMathOperator{\sgn}{sgn}

\usepackage[initials]{amsrefs}
 
\usepackage[
pdftitle={Index growth not imputable to topology},
pdfauthor={Alessandro Carlotto, Mario B. Schulz, David Wiygul}, 
pdfborder={0 0 0},
]{hyperref}

\newtheoremstyle{plain}
  {\topsep}   
  {0}         
  {\itshape}  
  {0pt}       
  {\bfseries} 
  {.}         
  {5pt plus 1pt minus 1pt} 
  {}          
  
\newtheoremstyle{remark}
  {\topsep}   
  {0}         
  {\normalfont}
  {0pt}       
  {\itshape}  
  {.}         
  {5pt plus 1pt minus 1pt} 
  {}          
  
\theoremstyle{plain}
\newtheorem{thm}{Theorem}[section]
\newtheorem{pro}[thm]{Proposition}
\newtheorem{lem}[thm]{Lemma}
\newtheorem{cor}[thm]{Corollary}

\theoremstyle{definition}
\newtheorem{xmpl}[thm]{Example}

\theoremstyle{remark}
\newtheorem{rem}[thm]{Remark}

\title{Index growth not imputable to topology}
\author{Alessandro Carlotto, Mario B. Schulz, David Wiygul}

\date{\vspace*{-4ex}}

\newcommand\printaddress{{
\setlength{\parindent}{17pt}
\par\bigskip
\bigskip
\vbox{
{\scshape Alessandro Carlotto}
\newline Universit\`a di Trento, 
Dipartimento di Matematica,
via Sommarive 14, 
38123 Povo di Trento, 
Italy
\newline
\textit{E-mail address:} 
\texttt{alessandro.carlotto@unitn.it}
\par\medskip
{\scshape Mario B. Schulz}
\newline Universit\`a di Trento, 
Dipartimento di Matematica,
via Sommarive 14, 
38123 Povo di Trento, 
Italy
\newline
\textit{E-mail address:} 
\texttt{mario.schulz@unitn.it}
\par\medskip
{\scshape David Wiygul}
\newline Universit\`a di Trento, 
Dipartimento di Matematica,
via Sommarive 14, 
38123 Povo di Trento, 
Italy
\newline
\textit{E-mail address:} 
\texttt{davidjames.wiygul@unitn.it}
\par
}
}}

\begin{document}

\maketitle

\begin{abstract}
We employ partitioning methods, in the spirit of Montiel--Ros but here recast for general actions of compact Lie groups, to prove effective lower bounds on the Morse index of certain families of closed minimal hypersurfaces in the round four-dimensional sphere, and of free boundary minimal hypersurfaces in the Euclidean four-dimensional ball. Our analysis reveals, in particular, phenomena of linear index growth for sequences of minimal hypersurfaces of fixed topological type, in strong contrast to the three-dimensional scenario.
\end{abstract}

\section{Introduction}\label{sec:intro}

In recent years significant efforts have been devoted to comparing different ``measures of complexity'' of minimal hypersurfaces, both in the closed and in the free boundary case. 
As a good starting point for our discussion, we shall recall how it has been conjectured by Schoen (and then, in more general forms by Marques and Neves) that in any compact, boundaryless manifold the Morse index of any closed minimal hypersurface should grow at least linearly with the first Betti number of the hypersurface in question (see \cite{Nev14}). Over the years at least three different methodologies have been developed to obtain bounds of that type: averaged energy estimates on the Euclidean components of harmonic forms (\cite{Sav10, AmbCarSha18, AmbCarSha18fbms}), Montiel--Ros partitioning methods for the lower spectrum of Schr\"odinger operators (\cite{MontielRos, KapouleasWiygulIndex, CarSchWiy23}) and covering arguments ``at the scale of the index'' (\cite{GrigNetrYau, Son23, CorFra24}); we refer the reader to the bibliography of such references for further contributions to the theme, especially for what concerns the special setting of complete minimal surfaces in $\R^3$.

 Although this conjecture is still open in its full generality, some of the aforementioned contributions have shed light on  certain important special instances. This is, in particular, the case for closed minimal surfaces in the round three-dimensional sphere (henceforth denoted $\Sp^3$) as established by Savo \cite{Sav10},  and for free boundary minimal surfaces in the Euclidean ball (henceforth denoted $\mathbb{B}^3$) thanks to the work of the first-named author with Ambrozio and Sharp \cite{AmbCarSha18fbms} and of Sargent \cite{Sar17}.

The converse inequality is, at least for closed minimal surfaces in three-dimensional manifolds of positive Ricci curvature, comparatively well-understood: Eijiri--Micallef \cite[Theorem 4.3]{EjiMic08} proved that having a bound on the topology \emph{and on the area} implies a bound on the Morse index, i.\,e.
\[
\Ind(\Sigma)\leq\kappa\bigl(\Area(\Sigma)+\genus(\Sigma)-1\bigr)
\]
where $\kappa$ is a computable numerical constant; hence, in combination with earlier work by Choi--Wang \cite[Proposition 4]{ChoWan83} ensuring that $\Area(\Sigma)\leq 16\pi (\genus(\Sigma)+1)/\lambda$ (where $\lambda$ denotes a positive lower bound for the Ricci curvature of the ambient manifold in question) one derives \emph{a linear upper bound for the Morse index in terms of the genus}. In fact, analogous results hold true for free boundary minimal surfaces, under suitable curvature assumptions, by virtue of the work of Lima \cite{Lim17}. In particular, it turns out that in both model cases, $\Sp^3$ and $\mathbb{B}^3$, we do in fact have two-sided bounds.

How about the higher-dimensional scenario? The goal of the present article is to show that even in the simplest possible ambient manifolds (i.\,e.~space forms) there exist minimal hypersurfaces having fixed topological type and arbitrarily large Morse index. 
Our analysis is quantitative, and exploits the Montiel--Ros methodology as developed in our previous work \cite{CarSchWiy23} (itself inspired by \cite{MontielRos, KapouleasWiygulIndex}), with the main results here recast in even greater generality, so as to handle also the case of continuous group actions (more generally and precisely: those associated to compact Lie groups). 
Such tools, which may well be of independent interest and ample applicability, are the object of Section \ref{sec:theory}.

The analysis of Hsiang's \cite{Hsiang1983} minimal hypersurfaces in $\Sp^4$ is the main object of Section \ref{sec:closed}. These hypersurfaces can be parametrized by the number of connected components in the intersection with the ``Clifford football'' in $\Sp^4$ (as defined in Section \ref{sec:closed}).
If we let $\hsiang_{m}$ denote the hypersurface with $m$ tori of intersection, then one can prove that the index grows at least linearly with $m$, with explicit estimates (see Theorem~\ref{thm:mainHyperspheres}).
In particular, in the case when the integer $m$ is odd, our result refers to the sequence of minimal hyperspheres produced by Hsiang to disporove Chern's spherical Bernstein conjecture. 

Lastly, in Section \ref{sec:fbms} we rather focus on the study of the Morse index of free boundary minimal hypersurfaces in the four-dimensional Euclidean ball $\B^4$, and specifically on the equivariant examples obtained by Freidin--Gulian--McGrath in \cite{FGM}, which in turn are suitable portions of the complete examples obtained by Alencar in \cite{Alencar}; see the statement of Theorem \ref{thm:mainTori} therein, which summarizes the outcome of our analysis. The aforementioned hypersurfaces are all solid tori, i.\,e.~they are all differomorphic to $\B^2\times\Sp^1$.

As will be apparent from the sequel of this manuscript, our arguments can easily be adapted to handle a large number of other examples in higher-dimensional round spheres and Euclidean balls, with fairly simple modifications. We note, incidentally, that the results we obtain have a fairly neat interpretation in terms of Choe's vision number \cite{ChoeVision90}. 
We further note that the minimal hypersurfaces in any of the sequences we consider have uniformly bounded area (volume) and so the unboundedness of the corresponding values of the Morse indices follows by appealing to Sharp's compactness theorem \cite{Sha17} or, for the free boundary case, from \cite{AmbCarSha18comp}. That being said, our methods here aim to provide much more refined information, both concerning the ``asymptotic growth rate'' of the index, and about the (possibly sharp) lower bounds, which may in principle be an important ingredient towards classification results e.\,g.~for the first few elements of the sequences in question.

All in all, while there remains the (hard) question of understanding the actual realm of validity of the Schoen--Marques--Neves conjecture, it thus clearly emerges that in ambient manifolds of dimension at least four the growth of the Morse index of minimal hypersurfaces may be imputable to different causes, other than the sole topological complexity. A deeper and more general understanding of these phenomena is an interesting challenge for the future.

\paragraph{Acknowledgments.} 
The authors thank the anonymous referee for their valuable suggestions. 
This project has received funding from the European Research Council (ERC) under the European Union’s Horizon 2020 research and innovation programme (grant agreement No.~947923). 
The research of M.\,S.~was funded by the Deutsche Forschungsgemeinschaft (DFG, German Research Foundation) under Germany's Excellence Strategy EXC 2044 -- 390685587, Mathematics M\"unster: Dynamics--Geometry--Structure, and the Collaborative Research Centre CRC 1442, Geometry: Deformations and Rigidity.

\section{Index estimates via partitioning methods for general group actions}\label{sec:theory}

Let $M$ be a smooth compact connected Riemannian manifold
(equipped with an unnamed metric),
with (possibly empty) boundary $\partial M$,
and let $\Omega$ be an open subset of $M$
with (possibly empty) Lipschitz boundary $\partial \Omega$.
We assume that $M$ has dimension at least one.
We will write $\eta$ 
for the outward unit normal to $\partial \Omega$,
defined almost everywhere thereon
with respect to the $(\dim(M)-1)$-dimensional Hausdorff measure
induced by the metric on $M$.

We will consider boundary-value problems on $\Omega$
for the Schr\"{o}dinger operator
$\Delta+q$, with $q$ a smooth real-valued function on $M$.
In general the boundary data will be mixed:
partially homogeneous Dirichlet,
partially homogeneous Neumann,
and partially Robin (or oblique).
To express these conditions
we assume the existence
of three pairwise disjoint
open subsets
$
  \dbdy\Omega,\nbdy\Omega,\rbdy\Omega
  \subseteq \partial \Omega
$
whose closures cover $\partial \Omega$.
The Robin condition,
to be applicable on $\rbdy\Omega$,
will be specified by a smooth function
$r\colon M\to\R$.
Momentarily we will give a precise,
weak formulation of the intended boundary value problem,
but for a sufficiently smooth solution $u$
the Robin condition will require
$\eta u = ru$ at almost every point of $\rbdy\Omega$.
(The homogeneous Neumann condition
could also be enforced via selection of $r$
and $\rbdy\Omega$,
but for our purposes
it is more convenient
to impose the homogeneous Neumann condition
explicitly on $\nbdy\Omega$.)

Working with the $\dim(M)$-dimensional
Hausdorff measure on $\Omega$ induced by the metric on $M$,
we write $L^2(\Omega)$ and $H^1(\Omega)$
for the standard Hilbert spaces
of (equivalence classes of) complex-valued functions on $\Omega$
which are square integrable
and which, for the latter space,
have square integrable derivatives,
relative to the metric on $M$.
Similarly,
now working with the $(\dim(M) - 1)$-dimensional
Hausdorff measure on $\partial\Omega$ induced by the metric on $M$,
we define $L^2(\rbdy\Omega)$.
We further set
\begin{equation*}
  H^1_{\dbdy\Omega}(\Omega)
  \vcentcolon=
  \{u \in H^1(\Omega) \st u|_{\dbdy\Omega}=0\},
\end{equation*}
where $u|_{\dbdy\Omega}$ is the restriction to $\dbdy\Omega$
of the trace of $u$.

Regarding $M$ (including its metric) along
with the functions $q$ and $r$ as fixed
(rather than as in \cite{CarSchWiy23}*{\S\,2.3},
where we employed more verbose notation,
indicating also these last data,
which we suppress here),
we associate to the data
$(\Omega, \dbdy\Omega, \nbdy\Omega, \rbdy\Omega)$
the sesquilinear form
\begin{equation}
\label{eqn:sesqui-def}
\begin{gathered}
  T[\Omega;\dbdy\Omega,\nbdy\Omega,\rbdy\Omega]
    \colon
    H^1_{\dbdy\Omega}(\Omega) \times H^1_{\dbdy\Omega}(\Omega)
    \to
    \C
  \\
  (u,v)
    \mapsto
    \sk[\big]{u,v}_{H^1(\Omega)}
      -\sk[\big]{u,(1+q)v}_{L^2(\Omega)}
      -\sk[\big]{u|_{\rbdy\Omega}, rv|_{\rbdy\Omega}}_{L^2(\rbdy\Omega)}.
\end{gathered}
\end{equation}
We will abbreviate this form by $T$
as long as there is no danger of confusion
as to the underlying data.

We call $\lambda \in \C$ an eigenvalue of $T$
if there exists nonzero $u \in H^1_{\dbdy\Omega}(\Omega)$
such that for all $v \in H^1_{\dbdy\Omega}(\Omega)$
there holds $T(u,v)=\lambda \sk{u,v}_{L^2(\Omega)}$,
in which case $u$ is called an eigenfunction of $T$
with eigenvalue $\lambda$.
Given an eigenvalue $\lambda$,
the span of all eigenfunctions of $T$ with eigenvalue $\lambda$
is called the eigenspace of $T$ with eigenvalue $\lambda$.
We have a spectral theorem for $T$:
the set of eigenvalues of $T$
is a discrete subset of $\R$
bounded below and unbounded above,
each eigenspace is finite-dimensional,
the eigenspaces are pairwise orthogonal in $L^2(\Omega)$,
and the eigenfunctions are dense
in $L^2(\Omega)$ and $H^1_{\dbdy\Omega}(\Omega)$.

For any $t \in \R$
and for any binary relation
$\mathord\sim \in \{\mathord=,\mathord<,\mathord\leq,\mathord>,\mathord\geq\}$
we write $E^{\sim t}(T)$ for the closure in $L^2(\Omega)$
of the span of all eigenfunctions of $T$
with eigenvalue $\lambda \sim t$,
and we write $\pi_T^{\sim t}$ for the $L^2(\Omega)$ projection
onto the closed subspace $E^{\sim t}$.

Now let $G$ be a compact Lie group acting on $M$.
We assume that the action
$G \times M \ni (g,p) \mapsto g.p \in M$
is smooth and via isometries of $M$
and that $q(g.p)=q(p)$ and $r(g.p)=r(p)$
for all $g \in G$ and $p \in M$.
We further assume
that $G$ preserves (setwise)
each of the sets $\Omega$,  $\dbdy\Omega$, $\nbdy\Omega$, and $\rbdy\Omega$.
Additionally we let
$\sigma \colon G \to \Ugroup(1)$
be a Lie group homomorphism from $G$
to the unitary group on $\C$.

We will say that a function $u \in L^2(\Omega)$
is $(G,\sigma)$-invariant if
$\sigma(g)u(g^{-1}.p)=u(p)$
for all $g \in G$ and almost every $p \in M$.
(In the important special case that $\sigma \equiv 1$
is the trivial homomorphism
we say $G$-invariant rather than $(G,1)$-invariant.)
Then the subspace of $(G,\sigma)$-invariant
functions in $L^2(\Omega)$ is closed
and the $L^2(\Omega)$ projection
$\pi_{G,\sigma}$ onto it is given by
\begin{equation*}
  (\pi_{G,\sigma}u)(p)
  =
  \int_G \sigma(g)u(g^{-1}.p) \, dg,
\end{equation*}
where $dg$ indicates
the left-invariant Haar measure
on $G$ for which $\int_G \, dg = 1$.
Moreover,
$\pi_{G,\sigma}|_{H^1(\Omega)}$
is the $H^1(\Omega)$ projection
onto the subspace of $(G,\sigma)$-invariant functions
in $H^1(\Omega)$, there holds
\begin{equation*}
  T(v,\pi_{G,\sigma}u)
  =
  T(u,\pi_{G,\sigma}v)
  \mbox{ for all } u,v \in H^1_{\dbdy\Omega}(\Omega),
\end{equation*}
and $\pi_{G,\sigma}$ commutes with each $\pi_T^{\sim t}$.
Given any $t \in \R$ and $\mathord\sim\in\{\mathord=,\mathord<,\mathord\leq,\mathord>,\mathord\geq\}$,
we define
\begin{equation}
\label{eqn:eigenspan-def}
  E_{G,\sigma}^{\sim t}(T)
  \vcentcolon=
  \pi_{G,\sigma}(E^{\sim t}(T)).
\end{equation}

We recall from Remark 2.2 in \cite{CarSchWiy23} that the space of $(G,\sigma)$-invariant functions could be, without any additional assumptions on the group and the associated twisting homomorphisms, finite-dimensional
(and even zero-dimensional). As a result, one could observe a trivialization of the eigenspaces above, and thus the sequences of corresponding eigenvalues could become finite (possibly of cardinality zero).

\begin{xmpl}\label{ref:example}
Suppose (for this example) that $M$ is a two-sided, connected embedded hypersurface
with a global unit normal $\nu$
in some ambient Riemannian manifold $N$,
suppose that $G$ is a subgroup of the isometry group of $N$
each of whose elements preserves $M$ as a set,
and define $\sigma=\sgn_\nu\colon G \to \Ogroup(1)<\Ugroup(1)$ by
\begin{equation*}
  \sgn_\nu(g)
  =
  \begin{cases}
    1 &\mbox{if } g_*\nu=\nu
    \\
    -1 &\mbox{if } g_*\nu=-\nu.
  \end{cases}
\end{equation*}
\end{xmpl}
In this context we will write
$G$-equivariant in place of $(G,\sgn_\nu)$-invariant.
When $M$ is a minimal hypersurface in $N$,
possibly subject to a boundary condition,
and $T$ is the corresponding stability (or Jacobi form),
then by the $G$-equivariant index of $M$
we mean $\dim E_{G,\sgn_\nu}^{<0}(T)$.
In all of the applications to follow in this paper
(in Sections \ref{sec:closed} and \ref{sec:fbms})
the homomorphism $\sgn_\nu$,
for the groups $G$ that we consider,
will always be the trivial homomorphism,
and so the notions of $G$-equivariant
and $G$-invariant coincide.

Now let $\Omega_1 \subset \Omega$
be another Lipschitz domain of $M$.
We define 
\begin{equation*}
\begin{aligned}
\intbdy\Omega_1
  &\vcentcolon=
  \partial\Omega_1 \cap \Omega,
\qquad
&\extbdy\Omega_1
  &\vcentcolon=
  \partial \Omega_1 
    \setminus
    \overline{\intbdy\Omega_1},
\\[1ex]
\dint{\dbdy}\Omega_1
  &\vcentcolon=
  (\extbdy\Omega_1 \cap \dbdy\Omega)
    \cup \intbdy\Omega_1,
\qquad
&\nint{\dbdy}\Omega_1
  &\vcentcolon=
  \extbdy\Omega_1 \cap \dbdy\Omega,
\\[1ex]
\dint{\nbdy}\Omega_1
  &\vcentcolon=
  \extbdy\Omega_1 \cap \nbdy\Omega,
\qquad
&\nint{\nbdy}\Omega_1
  &\vcentcolon=
  (\extbdy\Omega_1 \cap \nbdy\Omega)
    \cup \intbdy\Omega_1,
\\[1ex]
\dint{\rbdy}\Omega_1
  &\vcentcolon=
  \extbdy\Omega_1 \cap \rbdy\Omega,
\qquad
&\nint{\rbdy}\Omega_1
  &\vcentcolon=
  \extbdy\Omega_1 \cap \rbdy\Omega,
\end{aligned}
\end{equation*}
where $\overline{\intbdy\Omega_1}$
is the closure of $\intbdy\Omega_1$ in $\partial \Omega_1$.
We make these definitions
in order to pose
two different sets of boundary conditions on $\Omega_1$.
In both cases
$\partial\Omega_1$ inherits
whatever boundary condition is in effect
on $\partial\Omega$ wherever the two overlap
(corresponding to $\extbdy\Omega_1$).
The two sets of conditions are distinguished
by placing either
the Dirichlet or the Neumann condition
on the rest of the boundary
(corresponding to $\intbdy{\Omega_1}$).
Associated to these two sets of conditions
are the bilinear forms
\begin{equation}
\label{eqn:internalized-sesqui-def}
\begin{aligned}
\dint{T}_{\Omega_1}
  &\vcentcolon=
  T[
    \Omega_1;
    \dint{\dbdy}\Omega_1,
    \dint{\nbdy}\Omega_1,
    \dint{\rbdy}\Omega_1
  ],
\\[1ex]
\nint{T}_{\Omega_1}
  &\vcentcolon=
  T[
    \Omega_1;
    \nint{\dbdy}\Omega_1,
    \nint{\nbdy}\Omega_1,
    \nint{\rbdy}\Omega_1
  ],
\end{aligned}
\end{equation}
defined, respectively, on the Sobolev spaces
$H^1_{\dint{\dbdy}\Omega_1}(\Omega_1)$
and $H^1_{\nint{\dbdy}\Omega_1}(\Omega_1)$.

The following statement extends to compact Lie groups and complex-valued homomorphisms Proposition 3.1 in \cite{CarSchWiy23}, namely the most fundamental tool we employed to prove our index estimates,
in the more general terms anticipated in Remark 2.1 therein.

\begin{pro}
\label{pro:mr}
Let $M$ be a smooth compact connected Riemannian manifold
with (possibly empty) boundary.
Let $G$ be a compact Lie group
acting smoothly on $M$ by isometries,
and let $\sigma \colon G \to \Ugroup(1)$
be a homomorphism of Lie groups.
Let $q,r \colon M \to \R$ be smooth functions invariant under $G$.
Let $\Omega \subset M$ be a connected open subset of $M$,
and let $\Omega_1,\ldots,\Omega_n \subset \Omega$
be pairwise disjoint open subsets of $M$
whose closures cover $\Omega$.
Assume further that each of the preceding $n+1$ open sets
has Lipschitz boundary
(possibly with $\partial\Omega = \emptyset$)
and is preserved as a set by every element of $G$.
Let $\dbdy\Omega, \nbdy\Omega, \rbdy\Omega \subseteq \partial \Omega$
be pairwise disjoint open subsets of $\partial\Omega$
whose closures cover $\partial\Omega$
and each of which is preserved as a set by every element of $G$.

Define $T\vcentcolon=T[\Omega;\dbdy\Omega,\nbdy\Omega,\rbdy\Omega]$
as in \eqref{eqn:sesqui-def},
and for each $\Omega_i$,
define $\dint{T}_{\Omega_i}$ and $\nint{T}_{\Omega_i}$
as in \eqref{eqn:internalized-sesqui-def}.
Then, recalling \eqref{eqn:eigenspan-def}
for the definition of the following spaces of functions,
for every $t \in \R$ we have
\begin{enumerate}[label={\normalfont(\roman*)}]
\item \label{mrLower}
$\displaystyle
 \dim E_{G,\sigma}^{<t}(T)
 \geq
 \dim E_{G,\sigma}^{<t}(\dint{T}_{\Omega_1})
   +\sum_{i=2}^n \dim E_{G,\sigma}^{\leq t}(\dint{T}_{\Omega_i})
$, 
\item \label{mrUpper}
$\displaystyle
 \dim E_{G,\sigma}^{\leq t}(T)
 \leq
 \dim E_{G,\sigma}^{\leq t}(\nint{T}_{\Omega_1})
   +\sum_{i=2}^n \dim E_{G,\sigma}^{<t}(\nint{T}_{\Omega_i})
$.
\end{enumerate}
\end{pro}

\begin{proof}
The proof of \cite{CarSchWiy23}*{Proposition 3.1} goes through verbatim,
even with the assumptions on $(G,\sigma)$ generalized as above,
and working with vector spaces over $\C$ rather than $\R$.
\end{proof}

For the purposes of our later applications, it will also be convenient to spell out the following basic ``comparison'' result. We wish to stress that the conclusion holds true irrespective of any sign assumption on the Robin potential in play.

\begin{pro}\label{pro:RobDir}
Let $M$ be a smooth compact connected Riemannian manifold
with boundary $\partial M \neq \emptyset$ and with interior $\mathring{M}$.
Let $G$ be a compact Lie group
acting smoothly on $M$ by isometries,
and let $\sigma \colon G \to \Ugroup(1)$
be a homomorphism of Lie groups.
Let $q,r \colon M \to \R$ be smooth functions invariant under $G$.
Define
\begin{equation*}
\begin{aligned}
  T_\dir
  &\vcentcolon=
  T[
    \mathring{M};
    \dbdy\mathring{M}=\partial M,
    \nbdy\Omega=\emptyset,
    \rbdy\Omega=\emptyset
  ],
  \\
  T_\rob
  &\vcentcolon=
  T[
    \mathring{M};
    \dbdy\mathring{M}=\emptyset,
    \nbdy\Omega=\emptyset,
    \rbdy\Omega=\partial M
  ]
\end{aligned}
\end{equation*}
as in \eqref{eqn:sesqui-def},
and let $\lambda \in \R$.
Then, recalling \eqref{eqn:eigenspan-def}
for the definition of the following spaces of functions,
we have
\begin{equation*}
  \dim E_{G,\sigma}^{< \lambda}(T_\rob)
  \geq
  \dim E_{G,\sigma}^{\leq \lambda}(T_\dir).
\end{equation*}
\end{pro}

\begin{proof}
By the commutativity
of $\pi_{T_\rob}^{< \lambda}$ and $\pi_{G,\sigma}$ 
we have 
\begin{equation*}
  \pi_{T_\rob}^{< \lambda}\bigl(E_{G,\sigma}^{\leq \lambda}(T_\dir)\bigr)
  \subseteq
  E_{G,\sigma}^{< \lambda}(T_\rob).
\end{equation*}
It suffices therefore to check
that the restriction of $\pi_{T_\rob}^{< \lambda}$
to $E_{G,\sigma}^{\leq \lambda}(T_\dir)$
is injective.
Suppose then that $u \in E_{G,\sigma}^{\leq \lambda}(T_\dir)$
is orthogonal in $L^2(\mathring{M})$
to $E^{<\lambda}_{G,\sigma}(T_\rob)$,
and so in fact orthogonal to all of $E^{<\lambda}(T_\rob)$
(given that $u$ itself is $(G,\sigma)$-invariant,
so orthogonal to the orthogonal complement of
$\pi_{G,\sigma}(L^2(\mathring{M}))$).
Then $u \in E^{\geq \lambda}(T_\rob)$, so
\begin{equation*}
  \sk{u,u}_{H^1(\mathring{M})}
    -\sk{u,(1+q)u}_{L^2(\mathring{M})}
    -\sk{u|_{\partial M}, ru|_{\partial M}}_{L^2(\partial M)}
  =
  T_\rob(u,u)
  \geq
  \lambda \sk{u,u}_{L^2(\mathring{M})},
\end{equation*}
but, since
$
  u
  \in
  E_{G,\sigma}^{\leq \lambda}(T_\dir)
  \subset
  H^1_{\partial M}(\mathring{M})
$,
simultaneously
\begin{equation*}
  T_\rob(u,u)
  =
   \sk{u,u}_{H^1(\mathring{M})}
    -\sk{u,(1+q)u}_{L^2(\mathring{M})}
  =
  T_\dir(u,u)
  \leq
  \lambda \sk{u,u}_{L^2(\mathring{M})}.
\end{equation*}
Thus $T_\rob(u,u)=T_\dir(u,u)=\lambda \sk{u,u}_{L^2(\mathring{M})}$,
and in fact
$u \in E^{=\lambda}(T_\dir) \cap E^{=\lambda}(T_\rob)$.
It follows that $u$ is smooth
and satisfies $(\Delta+q+\lambda)u=0$
on $\mathring{M}$ and $u = \eta u = 0$ on $\partial M$.
From this Cauchy data on the boundary
and the unique continuation principle
we conclude that $u=0$,
ending the proof.
\end{proof}

\section{Applications to closed minimal hypersurfaces in round spheres}\label{sec:closed}

To put things in context, let us start with some recollections about what is known on closed (\emph{embedded}) minimal hypersurfaces in the round four-dimensional sphere $\Sp^4$. 
Zhou \cite{Zhou20} proved in 2020 the so-called \emph{multiplicity one conjecture} (crucially relying on the earlier work of many, and primarily on the rich theory developed by Marques and Neves), which implies that $\Sp^4$ contains a sequence $\{\Sigma_i\}_{i\in\N}$ of closed minimal hypersurfaces, having area (volume) diverging as $i\to\infty$ precisely as prescribed by the Weyl law of Liokumovich--Marques--Neves \cite{LioMarNev18}, but with only limited topological control (if any) and no effective lower bounds on the Morse index of such hypersurfaces. 
That said, to the best of our knowledge, only very few ``explicit'' examples are known. Among these, we first mention the equatorial spheres $\Sp^3$ that have Morse index equal to 1, and the Clifford-type products $\Sp^2(\sqrt{2/3})\times \Sp^1(\sqrt{1/3})$ (that have Morse index equal to 6) plus the two associated equivariant families constructed by Hsiang in \cite{Hsiang1983}, to be described below. To move beyond such instances, we shall recall Cartan's work \cite{Cartan1939} on the construction and classification of cubic isoparametric minimal
hypersurfaces in spheres; it turns out that there is precisely one in $\Sp^4$, and that is topologically a quotient
of $\SOgroup(3)$ by a suitable action of
$\mathbb{Z}_2\times\mathbb{Z}_2$. Solomon then proved in \cite{Sol90} that its Morse
index is equal to $20$: 
this seems to be the only case, besides equatorial hyperspheres and Clifford-type hypersurfaces, where this invariant has been computed explicitly. 
More recently, the first and second author of the present paper proved in \cite{CarSch21} the existence of one (conjecturally unique) minimal hypertorus in $\Sp^4$ (together with infinitely many immersed ones); the problem of computing its Morse index is open, but we shall note here that -- as a corollary of \cite[Theorem~4.3]{PerdomoLowIndex} -- such a hypertorus has Morse index at least $8$; we will get back to this matter later in this section, see in particular Remark \ref{rem:MarioNumerics} below.

That said, let $\torgrp$ denote the group
(isomorphic to $\Ogroup(2) \times \Ogroup(2)$)
of all isometries of $\Sp^4 \subset \R^5$
preserving (setwise) each of the great circles
$\{x^3=x^4=x^5=0\}$ and $\{x^1=x^2=x^5=0\}$.
Each of the two points $(0,0,0,0,\pm 1)$
is fixed by $\torgrp$.
These two antipodal points
lie on each of the great spheres
$\{x^3=x^4=0\}$ and $\{x^1=x^2=0\}$;
regarding these points as poles for these two spheres,
each of the corresponding circles of latitude
(having constant $x^5$)
is also an orbit of $\torgrp$.
Every other orbit of $\torgrp$
(aside from the preceding two fixed points
and two families of circular orbits)
is a torus.

Define also in $\Sp^4$ the $\torgrp$-invariant set
\begin{equation*}
  \football
  \vcentcolon=
  \bigl\{
    (x^1,x^2,x^3,x^4,x^5) \in \Sp^4
    \st
    (x^1)^2+(x^2)^2 = (x^3)^2+(x^4)^2
  \bigr\},
\end{equation*}
which, away from the two singular points $(0,0,0,0,\pm 1)$,
is an embedded minimal hypersurface; this is the so-called \emph{Clifford football} we alluded to in the introduction.

In \cite{Hsiang1983}
Hsiang constructed for each integer $m \geq 1$
a closed embedded $\torgrp$-invariant minimal hypersurface $\hsiang_m$
whose intersection with $\football$
consists of $m$ tori
and which contains precisely two circular orbits of~$\torgrp$.
When $m$ is even, both of these circles lie
in either $\{x^1=x^2=0\}$ or $\{x^3=x^4=0\}$,
and $\hsiang_m$ is homeomorphic to $\Sp^2 \times \Sp^1$;
when $m$ is odd, each of the preceding two spheres
contains exactly one of the circular orbits on $\hsiang_m$,
which in this case is homeomorphic to $\Sp^3$. We also wish to stress that in Hsiang's construction $\hsiang_1$ precisely coincides with the only $\torgrp$-invariant totally geodesic hypersphere (corresponding to the $r$-bisector, cf. \cite[Section 4]{Hsiang1983}).

For each $m$ we pick on $\hsiang_m$ a unit normal
$\nu_m=(\nu_m^1,\nu_m^2,\nu_m^3,\nu_m^4,\nu_m^5)$,
and we write $\jop_m$ for the Jacobi operator of $\hsiang_m$.
The function $\nu_m^5$ is $\torgrp$-invariant.

\begin{lem}\label{lem:Savo}
For each integer $m \geq 2$
the set $\{\nu_m^i\}_{i=1}^5$
of functions on $\hsiang_m$
is linearly independent
and
each is an eigenfunction of $\jop_m$ with eigenvalue $-3$.
\end{lem}

\begin{proof}
It is well-known (see for example Corollary 2.2 in \cite{Sav10})
that for any two-sided minimal hypersurface
$\Sigma \subset \Sp^{n+1} \subset \R^{n+2}$,
the normal component (to $\Sigma$ in $\Sp^{n+1}$)
of any translational Killing field on $\R^{n+2}$
is always (when nontrivial) an eigenfunction of eigenvalue $-n$
for the Jacobi operator $\jop$ of $\Sigma$,
and the component functions of the unit normal
are linearly independent
whenever $\Sigma$ is not totally geodesic.
However, to be self-contained we note how both claims can be easily seen by considering the cone $\widehat{\Sigma}$, centered at the origin,
over $\Sigma$ in $\R^{n+2}$, as follows.

Write $\nu=(\nu^i)_{i=1}^{n+2}$
for a unit normal to $\Sigma$
and $\widehat{\jop}$ for the Jacobi operator of $\widehat{\Sigma}$.
Then $\widehat{\nu}(x)=\nu(x/\abs{x})$
is a unit normal for $\widehat{\Sigma}$,
and each component $\widehat{\nu}^i$
is a Jacobi field for $\widehat{\Sigma}$
which is constant in the radial direction.
Therefore
\begin{equation*}
  0
  =
  \left.\left(\widehat{\jop}\widehat{\nu}^i\right)\right|_{\Sigma}
  =
  \jop\nu^i - n\nu^i,
\end{equation*}
where the second term arises from the Ricci curvature of $\Sp^{n+1}$.

Next suppose that the set $\{\nu^i\}_{i=1}^{n+2}$
is not linearly independent.
Then there exists a constant vector field $X$ on $\R^{n+2}$
which is everywhere orthogonal to $\nu$
and therefore also everywhere orthogonal to $\widehat{\nu}$.
This means that everywhere on $\widehat{\Sigma}$
the translational Killing field $X$
is tangential to $\widehat{\Sigma}$.
Thus the flow of $X$ preserves $\widehat{\Sigma}$,
and so $\widehat{\Sigma}$ is a hyperplane,
meaning that $\Sigma$ itself is a totally geodesic sphere.

Finally,
if $\Sigma$ is a $\torgrp$-invariant
totally geodesic sphere in $\Sp^4$,
then each of the two points at distance $\pi/2$ from $\Sigma$
must be fixed by $\torgrp$,
but this means that $\Sigma$ can be only $\{x^5=0\}$,
whose intersection with $\football$ consists of a single torus,
and which cannot therefore be any $\hsiang_m$ with $m \geq 2$.
\end{proof}

For each $m$
we now pick one of the two circular $\torgrp$-orbits on $\hsiang_m$ (to be consistent say the one lying on $\left\{x^3=x^4=0\right\}$),
we define $\dist_m \colon \hsiang_m \to \R$
to be the intrinsically defined distance function
from this circle,
and we write $\maxdst_m$ for the value of $\dist_m$
on the other circular $\torgrp$-orbit,
that is the distance in $\hsiang_m$
between the two circular $\torgrp$-orbits it contains.
Then for each $s \in \interval{0,\maxdst_m}$
the set $\{\dist_m = s\}$ is a torus,
while $\{\dist_m \leq s\}$ and $\{\dist_m \geq s\}$
are both solid tori.

We define also the function
$\vartheta \colon \Sp^4 \setminus \{\abs{x^5}=1\} \to \R$
by 
\begin{equation*}
   \vartheta(x^1,x^2,x^3,x^4,x^5)
  \vcentcolon=
  \arcsin\sqrt{\frac{\sum_{i=3}^4 (x^i)^2}{\sum_{i=1}^4 (x^i)^2}}-\frac{\pi}{4},
\end{equation*}
taking values in
$\IntervaL{-\frac{\pi}{4},\frac{\pi}{4}}$;
thus $\vartheta$ is a normalized signed distance
within the sphere of constant $x^5$
to~$\football$.
Then $\vartheta$ is smooth with nowhere vanishing gradient away from
$\{x^3=x^4=0\} \cup \{x^1=x^2=0\}$.
Furthermore,
$\vartheta$ is constant (or undefined)
on each orbit of $\torgrp$, as well as
on each orbit of the projection
$\partial_5^\top$ of the Killing field $\partial_5$
onto $\Sp^4$.

\begin{lem}\label{lem:CritChar}
For each integer $m \geq 1$
the critical points of $\vartheta|_{\hsiang_m}$
are the nodal points of $\nu_m^5$.
\end{lem}

\begin{proof}
The $\torgrp$-invariance of $\hsiang_m$
forces $\nu_m$ to be proportional to $\partial_5^\top$
on the two circular orbits of $\torgrp$ in $\hsiang_m$.
Thus these circles are disjoint from the nodal set of $\nu^5$.
Elsewhere $\vartheta|_{\hsiang_m}$ is smooth,
and a point $p \in \hsiang_m$
is a critical point of $\vartheta|_{\hsiang_m}$
if and only if $\nu_m|_p$ is proportional to $\nabla \vartheta|_p$,
but $\nabla \vartheta|_p$
is orthogonal to $\partial_5^\top|_p$
as well as to the toroidal $\torgrp$-orbit through $p$,
establishing the claim.
\end{proof}

\begin{lem}\label{lem:intermediate}
Let $m \geq 2$ be an integer,
and suppose $0<r<t<\maxdst_m$
satisfy $\{\dist_m=r\},\{\dist_m=t\} \subset \football$
and $\{r<\dist_m<t\} \cap \football = \emptyset$.
Then there exists $s \in \interval{r,t}$ such that ${\nu_m^5|}_{\{\dist_m=s\}}=0$.
\end{lem}

\begin{proof}
The function $\vartheta|_{\hsiang_m}$ attains the same value, namely zero, at each point of the tori $\{\dist_m=r\}$ and $\{\dist_m=t\}$. 
Thus, let $\gamma\colon I\to H_m$ be a smooth path, connecting a point on the former torus to a point on the latter, crossing each $G$-orbit orthogonally. 
Then the composite map $\vartheta|_{\hsiang_m}\circ\gamma\colon I\to \R$ must have a critical point at some some $s_0\in I$; then $p=\gamma(s_0)\in H_m$ is clearly a critical point for the function $\vartheta|_{\hsiang_m}\colon H_m\to\R$ and so by  $G$-invariance it follows that the whole $G$-orbit through $p$, denoted $\{\dist_m=s\}$, consists of critical points for such a function. 
The conclusion follows by appealing to Lemma \ref{lem:CritChar}.
\end{proof}

\begin{pro}
\label{prop:Hsiang-at-or-below-minus-3}
Let $m \geq 1$ be an integer.
Then $\jop_m$ acting on the space of $\torgrp$-invariant functions
has at least $m-1$ eigenvalues (counted with multiplicity)
strictly less than $-3$,
and $-3$ is itself an eigenvalue with multiplicity at least $1$;
if $m \geq 2$ and the restriction to $\torgrp$-invariant functions
is lifted, then the multiplicity of $-3$ is at least $5$.
\end{pro}

\begin{proof}
By Lemma \ref{lem:Savo} the components of $\nu_m$
span a five-dimensional subspace of the eigenspace of eigenvalue $-3$ for $m\geq 2$ (else for $m=1$ one has a totally geodesic sphere). 
Furthermore, by Lemma~\ref{lem:intermediate} and the way we labelled $H_m$ the $\torgrp$-invariant function $\nu_m^5$
has at least $m$ nodal domains
(with smooth toroidal boundaries), at which stage the proof is completed thanks to Proposition~\ref{pro:mr}.
\end{proof}

The next lemmata aim at refining our analysis, towards an upper bound for the number of eigenvalues not exceeding the value $-3$.

\begin{lem}
For each integer $m \geq 1$ the function $\vartheta|_{\hsiang_m}$ does not have any critical points on $\football$.
\end{lem}

\begin{proof}
Otherwise $\hsiang_m$ would be tangential to $\football$ along a torus in $\football$, and so
(using the unique continuation principle for linear elliptic equations)
would coincide with $\football$.
\end{proof}

Since the function $\vartheta$ is $\torgrp$-invariant,
for each $m$ there is a unique function
$\widetilde{\vartheta}_m\colon \IntervaL{0,\maxdst_m} \to \R$
such that
$\vartheta(p)=\widetilde{\vartheta}(\dist_m(p))$
for all $p \in \hsiang_m$.
The set of critical points of $\widetilde{\vartheta}_m$
is the image under $\dist_m$
of the set of critical points of $\vartheta|_{\hsiang_m}$.

\begin{lem}
For any integer $m \geq 1$
and any critical point $s$ of $\widetilde{\vartheta}_m$
the values of $\widetilde{\vartheta}_m$
and its second derivative at $s$
have opposite signs.
\end{lem}

\begin{proof}
Suppose $s$ is a critical point of $\widetilde{\vartheta}_m$,
and so let $p$ be a critical point of $\vartheta|_{\hsiang_m}$
for which $\dist_m(p)=s$.
Then $\nu_m|_p$ is tangential
to the constant-mean-curvature $3$-sphere $S$
of constant $x^5$ through $p$.
The sphere $S$ is preserved by $\torgrp$,
whose orbits thereon are mostly tori,
each of constant distance from $S \cap \football$ (itself an orbit),
degenerating to two circular orbits
at maximal distance from $S \cap \football$.
The areas of such tori are directly checked to increase toward $\football$,
and consequently the mean curvature in $S$
of $S \cap \hsiang_m$ points away from $\football \cap S$,
so at $p$ is proportional to $\nabla \vartheta|_p$,
with the sign of the constant of proportionality
agreeing with that of $\vartheta(p)$.
The claim now follows from the minimality of $\hsiang_m$ in~$\Sp^4$.
\end{proof}

\begin{cor}
Let $m \geq 1$ be an integer.
The function $\widetilde{\vartheta}_m$
has exactly one critical point
between any two consecutive zeros
and has no critical point
less than its smallest zero
or greater than its greatest zero. Hence,
for each integer $m \geq 2$
there exists a strictly increasing finite sequence
$\{s_{m,i}\}_{i=1}^{m-1} \subset \interval{0,\maxdst_m}$
of $m-1$ reals
such that
\begin{equation*}
  \{\nu_m^5=0\}
  =
  \bigcup_{i=1}^{m-1} \{\dist_m=s_{m,i}\},
\end{equation*}
a union of $m-1$ disjoint,
$\torgrp$-invariant tori,
and there exists  a strictly increasing finite sequence
$\{t_{m,i}\}_{i=1}^{m-2} \subset \interval{0,\maxdst_m}$
of $m-2$ reals
such that for each $i$
\begin{equation*}
  s_{m,i} < t_{m,i} < s_{m,i+1}
 \quad \mbox{and} \quad
 \{\dist_m=t_{m,i}\} \subseteq \{d\nu_m^5=0\}.
\end{equation*}
\end{cor}

\begin{lem}
\label{lem:first_dirichlet}
Let $m \geq 2$ and $i \in \IntervaL{1,m-2}$ be integers.
\begin{enumerate}[label={\normalfont(\roman*)}]
\item{Let $\Omega$ be any of the following domains:
$\{\dist_m \leq s_{m,1}\}$ (a solid torus),
$\{\dist_m \geq s_{m,m-1}\}$ (another solid torus),
or
$\{s_{m,i} \leq \dist_m \leq s_{m,i+1}\}$
(a torus times a closed interval)
for some integer $i \in \IntervaL{1,m-2}$.
Then $\nu_m^5|_\Omega$
is a first eigenfunction of $\jop_m$ on $\Omega$
subject to the homogeneous Dirichlet boundary condition.}
\item{Let $\Omega$ be either
$\{s_{m,i} \leq \dist_m \leq t_{m,i}\}$
or
$\{t_{m,i} \leq \dist_m \leq s_{m,i+1}\}$
(each a torus times a closed interval).
Then $\nu_m^5|_{\Omega}$
is a first eigenfunction of $\jop_m$ on $\Omega$
subject to the homogeneous Neumann condition on
$\{\dist_m=t_{m,i}\}$
and to the homogeneous Dirichlet condition
on the remanining boundary torus.}
\end{enumerate}
\end{lem}

\begin{proof}
For both items, it suffices to notice that the restriction in question has constant sign on the interior of the domain in question.
\end{proof}

\begin{lem}
\label{lem:second_neumann}
Let $m \geq 3$ be an integer, and let $\Omega$ be any of the following domains:
$\{\dist_m \leq t_{m,1}\}$ (a solid torus),
$\{\dist_m \geq t_{m,m-2}\}$ (another solid torus),
or $\{t_{m,i} \leq \dist_m \leq t_{m,i+1}\}$
(a torus times a closed interval)
for some $i\in\{1,\ldots,m-3\}$.
Then $\nu_m^5|_{\Omega}$
is a second $\torgrp$-invariant eigenfunction of $\jop_m$ on $\Omega$
subject to the homogeneous Neumann boundary condition,
and moreover there are no other second eigenfunctions
in this class
(aside from rescalings of this one).
\end{lem}

\begin{proof}
The restriction $\nu_m^5|_\Omega$ is a $\torgrp$-invariant eigenfunction
on $\Omega$ whose nodal set consists of a single torus
$\{\dist_m=s_{m,j}\}$ for suitable 
$j\in\{1,\ldots,m-1\}$.
In particular (again e.\,g.~appealing to Proposition~\ref{pro:mr}) we know that the second eigenvalue is at most $-3$.
Note also that, without imposing $G$-invariance,
any first eigenfunction
(which incidentally must be $G$-invariant)
has a sign.
Let $u$ be a second $\torgrp$-invariant eigenfunction.
Then there exists some $s \in \interval{0,\maxdst_m}$
with $\{\dist_m=s\} \subset \Omega$
and $u|_{\{\dist_m=s\}}=0$
(since otherwise $u$ could not be orthogonal to the first eigenspace).
If $s \neq s_{m,j}$,
then, by producing an appropriate test function
via extension by zero,
we can derive a contradiction
to Lemma \ref{lem:first_dirichlet}.
In more detail,
by such extension we could produce
a test function $v$
for one of the boundary-value problems
on a domain $\Omega'$
considered in Lemma \ref{lem:first_dirichlet}
(there called $\Omega$)
such that $v$ has Rayleigh quotient at most $-3$
and vanishes on a nonempty open subset of $\Omega'$;
by the unique continuation principle $v$
cannot be an eigenfunction,
so by the min-max characterization of eigenvalues,
we would then conclude that the first eigenvalue
of $\Omega'$,
with the appropriate boundary conditions,
is strictly less than $-3$,
contradicting Lemma \ref{lem:first_dirichlet}.
Finally, since first eigenvalues have multiplicity one,
another (much similar) application of Lemma \ref{lem:first_dirichlet}
now shows that $u$ is a constant multiple of $\nu_m^5|_\Omega$.
\end{proof}

\begin{pro}\label{Prop:UpperBoundHsiang}
For each integer $m \geq 2$
the Jacobi operator $\jop_m$
acting on the space of $\torgrp$-invariant
functions on $\hsiang_m$
has at most $m$ eigenvalues (counted with multiplicity)
less than or equal to $-3$.
\end{pro}

\begin{proof}
By Lemma \ref{lem:second_neumann} we have a partition of $\hsiang_m$ into $m-1$ domains 
(all of whose boundary components are smooth tori),
for each of which $-3$ is a second Neumann eigenvalue, with multiplicity one. (Note that when $m=2$ there is no actual partition, i.\,e.~we are just considering $H_2$ itself and thus no boundary conditions come into play.)
Thus, exploiting item \ref{mrUpper} of Proposition \ref{pro:mr}, we have one domain accounting for a $+2$ contribution, and exactly $m-2$ domains accounting for a $+1$ contribution, which gives the claim.
\end{proof}

We now combine together the pieces of information collected in the previous statements to prove the main result of this section.

\begin{thm}\label{thm:mainHyperspheres} For each integer $m \geq 2$ the following holds:
\begin{enumerate}[label={\normalfont(\roman*)}]
\item\label{itm:firstHsiang}
The Jacobi operator $\jop_m$ acting on the space of $\torgrp$-invariant functions on $\hsiang_m$ (cf.~Example \ref{ref:example}) has exactly $m-1$ eigenvalues (counted with multiplicity) strictly less than $-3$, and $-3$ is an eigenvalue of multiplicity $1$. 
\item\label{itm:secondHsiang}
The Morse index of $H_m$ is at least $m+4$.
\end{enumerate}
\end{thm}

\begin{proof}
According to Proposition \ref{prop:Hsiang-at-or-below-minus-3},
$\jop_m$ acting on the space of $\torgrp$-invariant functions on $\hsiang_m$ has at least $m-1$ eingenvalues strictly less than $-3$, plus the eigenvalue $-3$ with multiplicity at least one (the corresponding eingenfunction being $\nu_m^5$); this fact together with the upper bound given by Proposition \ref{Prop:UpperBoundHsiang} proves \ref{itm:firstHsiang}. 
In particular, the $G$-equivariant Morse index of $\hsiang_m$ is at least $m$. 
If we lift the equivariance constraint, we still have a subspace of dimension at least $m-1$ corresponding to the eigenvalues strictly less than $-3$ plus, in direct sum, a five-dimensional subspace spanned by eigenfunctions with eigenvalue exactly equal to $-3$ (because of Lemma \ref{lem:Savo}). 
The conclusion of \ref{itm:secondHsiang} follows at once.
\end{proof}

\begin{rem}
We expect that the hypersurface $H_2$ is actually congruent to the Clifford product $\Sp^2(\sqrt{2/3})\times \Sp^1(\sqrt{1/3})$. If that is the case, then it follows in particular that $H_2$ has Morse index exactly equal to 6, thus saturating the well-known general lower bound for the index of non-equatorial minimal hypersurfaces in $\Sp^{n}$ (see e.\,g. \cite[Corollary 2.2]{Sav10}, cf. \cite{Urb90}) and so proving the sharpness of the estimate given in item \ref{itm:secondHsiang} of Theorem \ref{thm:mainHyperspheres} above.
\end{rem}

\begin{rem}[Low index examples]\label{rem:MarioNumerics}
Concerning the $\torgrp$-equivariant Morse index bounds, the result in \ref{itm:firstHsiang} of Theorem \ref{thm:mainHyperspheres} should be compared with the bound for the equatorial hyperpshere $H_1$ (a unique negative direction, corresponding to $\nu^5_1$), 
with the bound for the (conjecturally Clifford) minimal hypersurface $H_2$ (at least two, in fact most likely exactly two $G$-equivariant ``negative directions'') and, more interestingly, with the aforementioned hypertorus, henceforth denoted $T^3$, constructed in~\cite{CarSch21}. 
Numerical simulations we carried through indicate that its 
$\Ogroup(2)\times\Ogroup(2)$-equivariant index is equal to~$3$; 
heuristically, in the notation of \cite{CarSch21} the three negative directions for the Jacobi quadratic form arise from ``radial scaling'' relative to the center of the 
$r$-$\vartheta$-plane and ``translations'' in the $r$-$\vartheta$-plane (the natural quotient space of $\Sp^4$ under the group action, cf.~\cite{HsiLaw71}).
The first direction (and only it) also respects the prescribed reflections in the 
$r$-$\vartheta$-plane, thus the equivariant index of $T^3$ with respect to its 
maximal symmetry group is expected to be equal to $1$. 
\end{rem}

\section{Applications to free boundary minimal hypersurfaces in Euclidean balls}\label{sec:fbms}

Let us shift gears and move to the analysis of free boundary minimal hypersurfaces in $\B^4$. Here, the landscape is even sparser than in $\Sp^4$: we have an ``abstract'' existence result of infinitely many (pairwise non congruent) such hypersurfaces (see \cite{Wan24} and references therein), plus -- to the best of our knowledge -- just three explicit examples. 
These are the equatorial ball
(with Morse index $1$), the higher-dimensional catenoid (with Morse index $5$, see \cite{SSTZ21}) and the $\Ogroup(2)\times\Ogroup(2)$-invariant examples constructed in \cite{FGM} (whose Morse index we wish to study in this section).

So let $\torgrp$ denote the group
(isomorphic to $\Ogroup(2) \times \Ogroup(2)$)
of all isometries of $\R^4$ setwise preserving
each of the planes $\{x^3=x^4=0\}$ and $\{x^1=x^2=0\}$.
The origin is fixed by $G$,
and each point in either $\{x^3=x^4=0\}$
or $\{x^1=x^2=0\}$ has orbit a circle in the same plane;
every other orbit is a torus.

Define in $\R^4$ the $\torgrp$-invariant set
\begin{equation*}
  \ccone
  \vcentcolon=
  \{
    (x^1,x^2,x^3,x^4) \in \R^4
    \st
    (x^1)^2+(x^2)^2=(x^3)^2+(x^4)^2
  \},
\end{equation*}
that is a minimal cone in $\R^4$ with link a Clifford torus in the unit $3$-sphere. 
In \cite{Alencar} Alencar constructed a complete, embedded $\torgrp$-invariant minimal hypersurface $\alencar$,
homeomorphic to $\R^2 \times \Sp^1$,
containing precisely one circular orbit of $\torgrp$,
asymptotic to $\ccone$,
and having nontrivial intersection with $\ccone$
outside every compact subset of $\R^4$ (see the statement of Theorem 1.1 therein).

We pick a unit normal $\nu=(\nu^1,\nu^2,\nu^3,\nu^4)$
for $\alencar$, and we write $\jop$ for its Jacobi operator.
Then $x^i\nu_i$, the normal component of the conformal Killing field $x^i\partial_i$
generating dilations about the origin,
is a Jacobi field on $\alencar$, which means that $\jop(x^i\nu_i)=0$. 
(Here and below we shall employ the standard sum convention for repeated indices, always understood with respect to the background Euclidean metric.)
We further define $\dist \colon \alencar \to \R$
to be the intrinsically defined distance function
from the unique circular orbit of $\torgrp$ on $\alencar$;
then for each $s>0$ the set $\{\dist=s\}$ is a torus
and $\{\dist \leq s\}$ is a solid torus.

We define also the function
$\vartheta \colon \R^4 \setminus \{(0,0,0,0)\} \to \R$
to be the directed distance in the origin-centered $3$-sphere
of radius $\abs{x}$ from $\ccone$,
increasing toward $\{x^1=x^2=0\}$,
normalized by $\abs{x}$:
\begin{equation*}
  \vartheta(x^1,x^2,x^3,x^4)
  \vcentcolon=
  \arcsin
    \sqrt{\frac{\sum_{i=3}^4 (x^i)^2}{\sum_{i=1}^4 (x^i)^2}}
  -\frac{\pi}{4}.
\end{equation*}
Then $\vartheta$ is smooth with nowhere vanishing gradient
away from $\{x^3=x^4=0\} \cup \{x^1=x^2=0\}$,
and $\vartheta$ is constant on each orbit of $\torgrp$
and on each orbit of $x^i\partial_i$.

The following lemmata 
now follow easily from the preceding definitions,
in parallel to the analysis of the Hsiang hypersurfaces in $\Sp^4$ that was carried through in Section \ref{sec:closed}.

\begin{lem}
The critical points of $\vartheta|_\alencar$
are the nodal points of $x^i\nu_i$.
\end{lem}

\begin{lem}
There exists a strictly increasing, diverging sequence
$\{s_m\}_{m=1}^\infty$ of reals such that
\begin{equation*}
  \{x\in A \st x^k\nu_k(x)=0\} = \bigcup_{m=1}^\infty \{\dist=s_m\},
\end{equation*}
a union of disjoint $\torgrp$-invariant tori.
\end{lem}

We need just one more fact, which will ensure that each origin-centered sphere in $\R^4$ intersects $\alencar$ in at most one $\torgrp$-orbit.

\begin{lem}
The set of critical points of the restriction to $\alencar$ of the distance function $\abs{x}$ in $\R^4$ is $\{\dist=0\}$.
\end{lem}

\begin{proof}
First note that $\abs{x}|_{\alencar}$ is smooth,
since $\alencar$ omits the origin.
The $\torgrp$-invariance of $\alencar$
and the fact that $\{\dist=0\}$ is a circular $\torgrp$-orbit
uniquely determine the normal line in $\R^4$
to $\alencar$ at any point on $\{\dist=0\}$:
it is radial (intersects the origin).
Consequently, $\{\dist=0\}$
consists of critical points
of $\abs{x}|_{\alencar}$.

Next, at any (potential) critical point $y \in \alencar$ of $\abs{x}|_\alencar$ not belonging to  $\left\{\rho=0\right\}$
we have tangency of $\alencar$
to the sphere $\{x\in\R^4\st \abs{x}=\abs{y}\}$.
Since the $\torgrp$-orbit of $y$
is contained in this sphere,
it follows from the convexity of the latter
that the partial trace of the vector-valued second fundamental form
of $\alencar$ at $y$ over the tangent space
to the $\torgrp$-orbit
is nonzero and directed toward the origin.
The minimality of $\alencar$ then forces the value of $\abs{x}|_{\alencar}$
to be strictly increasing as one moves away from the point $y$ in the direction
orthogonal to the $\torgrp$-orbit. 

Now, consider instead a critical point $y$ on the $\left\{\rho=0\right\}$ orbit. Here $\alencar$ has vector-valued second fundamental form $\II_y(v,v)$ directed radially inward for $v$ tangent to the orbit (as for the toroidal orbits considered above). 
The two-dimensional subspace of $T_y A$ orthogonal to $v$ is invariant under the subgroup of $\torgrp$ that fixes $y$ (a copy of $\Ogroup(2)$). 
By the symmetry we have $\II_y(u,u)=\II_y(w,w)$ for any unit vectors $u$ and $w$ in this subspace. 
By the minimality of $A$ we must then have $\II_y(u,u)$ directed radially outward for every such $u$, and so also in this case $\abs{x}|_{\alencar}$ is strictly increasing as one moves away from the point $y$ in directions orthogonal to the $\torgrp$-orbit. 

Finally, if there were a critical point off $\{\dist=0\}$,
then (by virtue of the $\torgrp$-invariance)
we could find a smooth path on $\alencar$,
orthogonal to every $\torgrp$-orbit it intersects,
joining a critical point on $\{\dist=0\}$
to a critical point on another orbit.
By the conclusion of the preceding paragraphs
each endpoint is a strict local minimizer
of the restriction of $\abs{x}$ to the path in question,
and so, in between, by a standard (one-dimensional) min-max argument this restriction would have
a third critical point which is not a strict local minimizer,
in contradiction to the conclusion of the preceding paragraph.
\end{proof}

In fact, it follows at once from the previous argument that the restriction to $\alencar$ of the ambient distance from the origin is a (well-defined) monotone function of $\dist$, and since they both tend to $+\infty$, necessarily a nondecreasing one. Hence:

\begin{cor}
There exists a strictly increasing, diverging sequence
$\{r_m\}_{m=1}^\infty \subset \R$
such that
\begin{equation*}
\bigl\{x \in \alencar \st  x^k\nu_k(x)=0\bigr\}
  =
  \bigcup_{m=1}^\infty
    \{x \in \alencar \st \abs{x}=r_m\}, 
\end{equation*}
a union of disjoint $\torgrp$-invariant tori.
\end{cor}
In particular it follows that for each integer $\ell \geq 1$ the set
\begin{equation}
\label{eqn:fbmTorus}
\alencar_\ell\vcentcolon=
\frac{1}{r_\ell}\bigl(\alencar \cap \{\abs{x} \leq r_\ell\}\bigr)
\end{equation}
is a free boundary minimal solid torus in $\B^4$.
These surfaces are precisely the solid tori
$\Sigma_{2,2,\ell}$ of \cite{FGM}.

\begin{thm}\label{thm:mainTori}
For each integer $\ell\geq 1$ the free boundary minimal solid torus $\alencar_\ell \subset \B^4$
defined in \eqref{eqn:fbmTorus} 
has $\torgrp$-equivariant Morse index at least $\ell$.
\end{thm}

\begin{proof}
The $\torgrp$-invariant Jacobi field $x^k\nu_k$ has $\ell$ nodal domains
and satisfies the homogeneous Dirichlet condition on $\partial \alencar_\ell$.
It then follows (from item \ref{mrLower} of Proposition \ref{pro:mr}, applied for $q$ equal to the squared length of the second fundamental form of $A_{\ell}$) that the Jacobi operator for $\alencar_\ell$
acting on $\torgrp$-invariant functions (equivalently:~$\torgrp$-equivariant functions, cf.~Example~\ref{ref:example})
subject to the homogeneous Dirichlet condition
has at least $\ell-1$ strictly negative eigenvalues
(counted with multiplicity),
and so, including $x^k\nu_k$ itself,
at least $\ell$ nonpositive eigenvalues.
At this stage the conclusion comes by simply appealing to Proposition \ref{pro:RobDir}, which allows us to compare the preceding bound with the geometrically relevant one where oblique boundary conditions are imposed (for $r\equiv 1$).
\end{proof}

\begin{rem}\label{rem:asym}
We explicitly note that Proposition \ref{pro:RobDir}, applied to $A_{\ell}$
at any threshold level $\lambda\in\R$ (in lieu of $\lambda=0$), implies that the eigenvalues shift downward
in changing from the homogeneous Dirichlet
to the natural Robin boundary condition (irrespective of the sign of the Robin potential); for this reason we do not obtain an upper bound
in the same fashion as for the closed Hsiang hypersurfaces.
\end{rem}

\clearpage
\setlength{\parskip}{1ex plus 1pt minus 1pt}
\bibliography{minimalsurfaces-bibtex}

\printaddress

\end{document}